\documentclass{article}

\usepackage[utf8]{inputenc} 
\usepackage{xspace}
\usepackage{amsthm}
\usepackage{hyperref}
\usepackage{mathtools}
\usepackage[english]{babel} 
\usepackage{enumitem}
\usepackage{csquotes}
\usepackage{bussproofs}
\usepackage{bbold}

\usepackage[T1]{fontenc} 
\usepackage{lmodern} 
\usepackage{amssymb} 
\usepackage{stmaryrd} 
\usepackage{cmll}
\usepackage{xcolor}
\usepackage{authblk}

\usepackage{amsmath}
\usepackage{amssymb}
\usepackage{stmaryrd}
\usepackage{biblatex}
\usepackage{todonotes}
\usepackage{cmll}

\usepackage{tikz-cd}
\usetikzlibrary{arrows}
\usetikzlibrary{matrix}
\usetikzlibrary{arrows,positioning,calc}

\newtheorem{theorem}{Theorem}[section]
\newtheorem{lemma}[theorem]{Lemma}

\newtheorem{definition}[theorem] {Definition}

\newcommand{\Set}{\mathrm{Set}}
\newcommand{\Cat}{\mathrm{Cat}}

\newcommand*{\Rel}{\mathrm{Rel}}

\newcommand{\op}{\mathrm{op}}

\newcommand*{\1}{\mathbb 1}

\newcommand*{\C}{\mathbb C}

\newcommand*{\D}{\mathbb D}

\newcommand*{\todi}{\rlap{\ \ \raisebox{.5px}{$\shortmid$}}\rightarrow}

\makeatletter
\newcommand*{\doublerightarrow}[2]{\mathrel{
  \settowidth{\@tempdima}{$\scriptstyle#1$}
  \settowidth{\@tempdimb}{$\scriptstyle#2$}
  \ifdim\@tempdimb>\@tempdima \@tempdima=\@tempdimb\fi
  \mathop{\vcenter{
    \offinterlineskip\ialign{\hbox to\dimexpr\@tempdima+1em{##}\cr
    \rightarrowfill\cr\noalign{\kern.5ex}
    \rightarrowfill\cr}}}\limits^{\!#1}_{\!#2}}}
\newcommand*{\triplerightarrow}[1]{\mathrel{
  \settowidth{\@tempdima}{$\scriptstyle#1$}
  \mathop{\vcenter{
    \offinterlineskip\ialign{\hbox to\dimexpr\@tempdima+1em{##}\cr
    \rightarrowfill\cr\noalign{\kern.5ex}
    \rightarrowfill\cr\noalign{\kern.5ex}
    \rightarrowfill\cr}}}\limits^{\!#1}}}
\newcommand{\colim@}[2]{\vtop{\m@th\ialign{##\cr
    \hfil$#1\operator@font lim$\hfil\cr
    \noalign{\nointerlineskip\kern1.5\ex@}#2\cr
    \noalign{\nointerlineskip\kern-\ex@}\cr}}}
\newcommand{\colim}{%
  \mathop{\mathpalette\colim@{\rightarrowfill@\textstyle}}\nmlimits@
}
\makeatother

\newcommand\eA{\mathcal A}
\newcommand\eD{\mathcal D}
\newcommand\eV{\mathcal V}
\newcommand\eH{\mathcal H}
\newcommand\eE{\mathcal E}
\newcommand\A{\mathbb A}
\renewcommand\H{\mathbb H}
\newcommand\V{\mathbb V}
\newcommand\T{\mathbb T}
\newcommand\E{\mathbb E}

\newcommand\Span{\mathtt{Span}}

\author{Flavien Breuvart}
\affil
{Laboratoire LIPN, CNRS UMR 7030,Universit\'e Sorbone Paris Nord, France\\
\texttt{breuvart@lipn.univ-paris13.fr}
}

\begin{document}

\title{Reflection on the enrichments of double categories}
\maketitle

This version an extended but early draft : with hardly any references, explanations and discussions, lacking some definitions and examples, and, above all, without formal proofs. To be frank, some proofs have only be roughly checked on papers, and there may well be mistakes.

A more definitive version should follow.

\section{Preliminaries}

\begin{definition}[Double category]\ \\
  A double category $\D=(\A,\V,\H,\odot,U)$ is given by :
  \begin{itemize}
  \item a category $\V$ called vertical category,
  \item a category $\A$ called arrow category,
  \item two functors $S,T : \A\rightarrow \V$ called source and target, objects $u\in|\A|$ of source $X$ and target $Y$ are denoted as crossed arrow $u:X\todi Y$, 
  \item a functor $U : \V \rightarrow \A$ such that $U;S = U;T = id^{\Cat}_{\V}$,
  \item a functor $(\_\odot\_) : T\vee S \rightarrow \A$ such that $(\_\odot\_);S=S_{|T}S$ and $(\_\odot\_);T=T_{|S;T}$, where $T\vee S$ is the pullback of $S$ and $T$, $T_{|S}:T\vee S \rightarrow \A$ is the pullback of $T$ along $S$ and $S_{|T}:T\vee S \rightarrow \A$ is the pullback of $S$ along $T$,
  \item with natural isomorphisms :
    \begin{align*}
      \lambda &: \phi \simeq U_X\odot \phi &
      \rho &: \phi \simeq \phi\odot U_X &
      \alpha &: \phi\odot(\psi\odot\zeta) \simeq (\phi\odot\psi)\odot\zeta
    \end{align*}
    \item such that $S(\lambda)$, $S(\rho)$, $S(\alpha)$, $T(\lambda)$, $T(\rho)$ and $T(\alpha)$ are all identities,
    \item with the coherence axioms :
    \begin{align*}
      id\odot\lambda &= \alpha;(\rho\odot id) &
      \alpha;\alpha &= (id\odot\alpha);\alpha;(\alpha\odot id)
    \end{align*}      
  \end{itemize}
  We denote morphisms of $\A(u,v)$ of source $f$ and target $g$ by a square : 
    \begin{center}
      \begin{tikzpicture}
        \node (X) at (0,1.5) {$X$};
        \node (Y) at (1.5,1.5) {$Y$};
        \node (V) at (0,0) {$V$};
        \node (W) at (1.5,0) {$W$};
        \draw[->] (X) to node [sloped] {$\shortmid$} node[auto] {{\scriptsize $u$}}  (Y);
        \draw[->] (X) to node[auto] {{\scriptsize $f$}}   (V);
        \draw[->] (Y) to node[auto] {{\scriptsize $g$}}  (W);
        \draw[->] (V) to node [sloped] {$\shortmid$} node[auto] {{\scriptsize $u$}} (W);
        \node () at (.75,.75) {$\phi$};
      \end{tikzpicture}
    \end{center}
    Vertical compositions of such a square correspond to composition in $\A$ and horizontal compositions of such squares correspond to $\odot$-composition. up to $\lambda$, $\rho$ and $\alpha$ isomorphisms, the way we compose the diagrams do not matter.

    In addition, we define a few additional categories :
  \begin{itemize}
  \item we call horizontal category the category $\H$ :
    \begin{itemize}
    \item which objects are those of the vertical category : $|\H|=|\V|$,
    \item which morphisms are the objects of the arrow category : 
      $$\H(X,Y) = \{u\in|\A| \mid S(u)= X,\ T(u)= Y\}$$ 
    \item and which unit and composition are given by $U$ and $\odot$ on objects.
    \end{itemize}
  \item we call diagonal category the category $\T$ : 
    \begin{itemize}
    \item which objects are those of the vertical category : $|\H|=|\V|$,
    \item which morphisms are the tuples 
      $$\T(X,Y) :=\left\{(f,u,\phi,\psi) \middle | \begin{matrix} f\in\V(X,Y), u\in\T(X,Y), \phi\in\A(id_X^\H,u), \psi\in\A(u,id^\H_Y),\\ S(\phi)=id_X^\V,\ T(\psi)=id_Y^\V,\ T(\phi)=S(\psi)= f \end{matrix} \right\}$$
    \item and which unit and composition are given by :
      $$ id^\T_{X}:= (id^\V_X,U_X,U_{id^\V},U_{id^\V}) $$
      $$(f,u,\phi,\psi);(g,v,\phi',\psi') := (f;g,\ u\odot v;\ ((\phi;U_g)\odot(id^\A_v;\phi')),\ ((\psi;id^\A_u)\odot(U_f;\psi')) $$
    \end{itemize}
  \item we call globular horizontal category the 2-category $\H^2$ which objects and morphisms are those of $\H$ and which 2-morphisms are
    $$\H^2(u,v):= \{\phi\in \A(u,v) \mid S(\phi) = id^\V_{S(u)},\ T(\phi)=id^\V_{T(u)}\}$$
  \item we call globular vertical category the 2-category $\V^2$ the category with the which objects and morphisms are those of $\V$ and which 2-morphisms are
    $$\V^2(f,g):= \{\phi\in \A(id^\H_{S(f)},id^\H_{T(f)}) \mid S(\phi) = f,\ T(\phi)=g\}$$
  \item we call transversal double category $\D^t$ the category obtained by inverting the horizontal and vertical category, concretely :
    \begin{itemize}
    \item $\V^t =\H^t$
    \item $|\A^t| = \{(X,Y,f)\mid X,Y\in|\V|, f\in\V(X,Y)\}$,
    \item $\A^t((X,Y,f),(X',Y',g)) = \{(u,v,\phi) \mid u,v\in|\A|, \phi\in\A(u,v), S(\phi)= f, T(\phi)= g\}$,
    \item $id^{\A^t}_{X,Y,f}= (id_X,id_Y,U_f)$, $(u,v,\phi);^t(u',v',\psi) = (u\odot u',v\odot v',\phi\odot\psi)$,
    \item $U_u= (u,u,id^\A_u)$, $(u,v,\phi)\odot(u',v',\psi) = (u;u',v;v',\phi;\psi)$,
    \item $S^t(u,v,\phi)=u $, $T^t(u,v,\phi)=v$
    \end{itemize}
  \end{itemize}
  \item we call vertically opposite double category $\D^{vop}$ the double category obtained by taking the opposite of $\V^{op}$ of $\V$ and $\A^\op$ of $\A$.
\end{definition}

\begin{definition}[Monoidal double category]
  A double category $\eD =(\eA,\eH,\eV,U,\odot,\otimes,\1)$ is monoidal when :
  \begin{itemize}
  \item $(\eV,\otimes^\eV,\1^\eV)$ is monoidal,
  \item $(\eA,\otimes^\eA,\1^\eA)$ is monoidal,
  \item $S,T,U,\odot$ are strict monoidal functors,
  \end{itemize}
\end{definition}

\section{Existing concepts of enrichments or restrictions}

\begin{definition}[Externaly enriched double category]
  Let $\mathcal C$ a monoidal category with all pullbacks.

  An externally $\C$-enriched category is an internally category of the category of $\C$-enriched categories, i.e., it is given by :
  \begin{itemize}
  \item a $\C$-enriched category $\V$,
  \item a $\C$-enriched category $\A$,
  \item two $\C$-enriched functors $S,T : \A\rightarrow \V$, 
  \item a $\C$-enriched functor $U : \V \rightarrow \A$ such that $U;S = U;T = id^{\C\Cat}_{\V}$,
  \item a $\C$-enriched functor $(\_\odot\_) : T\vee S \rightarrow \A$ such that $(\_\odot\_);S=S_{|T}S$ and $(\_\odot\_);T=T_{|S;T}$, where $T\vee S$ is the pullback of $S$ and $T$, $T_{|S}:T\vee S \rightarrow \A$ is the pullback of $T$ along $S$ and $S_{|T}:T\vee S \rightarrow \A$ is the pullback of $S$ along $T$,
  \item with the same isomorphisms and coherence axioms
  \end{itemize}
\end{definition}

\begin{definition}[Thin double category]\ \\
  A thin (also called posetal) double category $\D$ is a double categories such that, for any given $f,g,u,v$, there is at most one square $\eta\in\A(u,v)$ such that $S(\eta)=f$ and $T(\eta)=g$.
\end{definition}

\begin{definition}[Inclusion double category]\ \\
  An inclusion double category $\D$ is a double categories which vertical category is included in the diagonal category :
  $$ \forall f\in\V(X,Y), \exists u,\phi,\psi,\ (f,u\phi,\psi)\in\T(X,Y) $$
\end{definition}

\begin{definition}[Framed bicategory]\ \\
  A framed bicategory $\D$ is an inclusion double category $\D$ which vertical posit $\D^{vop}$ is also an inclusion category.
\end{definition}

%
%
%

\section{internally enriched double categories}

As oppose to the external enrichments that internalize the notion of category in the 2-category of $\C$-enriches category, we are trying to internalized the notion of $|\C|$-enriched category into the 2-category of categories.

Intuitively, a $\mathcal C$-enriched category is given by a set $\V$ of objects, a function $\E : \V\times \V\rightarrow |\mathcal C|$, a unit $U_X: \mathcal C(\1,\E(X,X))$ and a composition $\odot_{x,y,z}:\mathcal C(\E(X,Y)\otimes \E(Y,Z), \E(X,Z))$, plus standard coherence diagrams.

The immediate generalization is to replace set by categories and functions by functors. However, such a notion is too restrictive due to the strictness of $\E$. Indeed, such a definition require $\E(id_X,id_Y)=id_{\E(X,Y)}$, which, intuitively, means that the globular horizontal category will be trivial. Similarly, the strictness on the composition do not make sense either.

However, one can use a lax functor for $\E$ and resolve this issue. Additionally, since we are working with double categories, we know that lax functors make way more sense in the context of double categories rather that 2-categories, thus we will require $\mathcal C$ to be a double category (that we call $\eD$). Finally, having access to a new ``kind'' of morphisms in $\eD$, we can distinguish those used as target of $\E$ and those used for the unit $U$ and composition $\odot$, allowing a richer framework.

\begin{definition}[horizontal enrichement]\ \\
  Let $\eD =(\eA,\eH,\eV,U^\eD,\odot^\eD,\otimes^\eD,\1^\eD)$ be a monoidal double category.

  An internally $\eD$-enriched double category $\D$ is given by :
  \begin{itemize}
  \item a vertical category $\V$
  \item a lax functor :
    $$ \E : \V\times\V \rightarrow \eH $$
    whose lax-ness spells out as globular cells :
    \begin{center}
      \hspace{-1cm}
      \begin{tikzpicture}
        \node (X) at (0,2) {$\E(X,Y)$};
        \node (B) at (2,2) {$\E(X,Y)$};
        \node (Z) at (2,0) {$\E(X,Y)$};
        \node (Y) at (0,0) {$\E(X,Y)$};
        \draw[double] (X) to node [sloped] {$\shortmid$} (B);
        \draw[double] (X) to  (Y);
        \draw[double] (B) to (Z);
        \draw[->] (Y) to node [sloped] {$\shortmid$} node[below] {{\scriptsize $\E(id_X^\V,id_Y^\V)$}} (Z);
        \node () at (1,1) {$I^\E_{X,Y}$};
      \end{tikzpicture}
      \begin{tikzpicture}
        \node (X) at (0,2) {$\E(X_1,Y_1)$};
        \node (X') at (2.5,2) {$\E(X_2,Y_2)$};
        \node (B) at (5,2) {$\E(X_3,Y_3)$};
        \node (Z) at (5,0) {$\E(X_3,Y_3)$};
        \node (Y) at (0,0) {$\E(X_1,Y_1)$};
        \draw[->] (X) to node [sloped] {$\shortmid$} node[above] {{\scriptsize $\E(f,g)$}} (X');
        \draw[->] (X') to node [sloped] {$\shortmid$} node[above] {{\scriptsize $\E(f',g')$}} (B);
        \draw[double] (X) to  (Y);
        \draw[double] (B) to (Z);
        \draw[->] (Y) to node [sloped] {$\shortmid$} node[below] {{\scriptsize $\E(f;f',g;g')$}} (Z);
        \node () at (2.5,1) {$C^\D_{f,g,f',g'}$};
      \end{tikzpicture}
    \end{center}
  \item with
    $$ U^\D_X \in \eV(\1,\E(X,X)) \quad\quad U^\D_f \in \eA(id_\1,\E(f,f))$$
    $$ (\_\odot^\D\_) \in \eV(\E(X,Y)\otimes \E(Y,Z) , \E(X,Z)) $$
    $$ (\_\odot^\D\_) \in \eA(\E(f,g)\otimes \E(g,h) , \E(f,h)) $$
    ``functorial'' in in the sens that (in $\E$) :
    \begin{center}
      \hspace{-1cm}
      \begin{tikzpicture}
        \node (X) at (0,2) {$\1$};
        \node (B) at (2,2) {$\1$};
        \node (Z) at (2,0) {$\E(X',X')$};
        \node (Y) at (0,0) {$\E(X,X)$};
        \draw[double] (X) to node [sloped] {$\shortmid$} (B);
        \draw[->] (X) to node[left] {{\scriptsize $U^\D_X$}} (Y);
        \draw[->] (B) to node[right] {{\scriptsize $U^\D_{X'}$}} (Z);
        \draw[->] (Y) to node [sloped] {$\shortmid$} node[below] {{\scriptsize $\E(f,f)$}} (Z);
        \node () at (1,1) {$U^\D_f$};
      \end{tikzpicture}
      \begin{tikzpicture}
        \node (X) at (0,2) {$\E(X,Y)\otimes\E(Y,Z)$};
        \node (B) at (5,2) {$\E(X',Y')\otimes\E(Y',Z')$};
        \node (Z) at (5,0) {$\E(X',Z')$};
        \node (Y) at (0,0) {$\E(X,Z)$};
        \draw[->] (X) to node [sloped] {$\shortmid$} node[above] {{\scriptsize $\E(f,g)\otimes\E(g,h)$}} (B);
        \draw[->] (X) to node[left] {{\scriptsize $\odot^\D_{X,Y,Z}$}} (Y);
        \draw[->] (B) to node[right] {{\scriptsize $\odot^\D_{X',Y',Z'}$}} (Z);
        \draw[->] (Y) to node [sloped] {$\shortmid$} node[below] {{\scriptsize $\E(f,h)$}} (Z);
        \node () at (2.5,1) {$\odot^\D_{f,g,h}$};
      \end{tikzpicture}
  \end{center}
    and the lax-functorial equalities :
    \begin{center}
      \begin{tikzpicture}
        \node (X) at (0,3) {$\1$};
        \node (B) at (2,3) {$\1$};
        \node (Z) at (2,0) {$\E(X,X)$};
        \node (Y) at (0,0) {$\E(X,X)$};
        \draw[double] (X) to node [sloped] {$\shortmid$} (B);
        \draw[->] (X) to node[left] {{\scriptsize $U^\D_X$}} (Y);
        \draw[->] (B) to node[right] {{\scriptsize $U^\D_{X}$}} (Z);
        \draw[->] (Y) to node [sloped] {$\shortmid$} node[below] {{\scriptsize $\E(id^\V_X,id^\V_X)$}} (Z);
        \node () at (1,1.5) {$U^\D_{id^\V_X}$};
        \node () at (3,1) {$=$};
      \end{tikzpicture}
      \begin{tikzpicture}
        \node (X) at (0,3) {$\1$};
        \node (B) at (2,3) {$\1$};
        \node (Z') at (2,1.5) {$\E(X,X)$};
        \node (Y') at (0,1.5) {$\E(X,X)$};
        \node (Z) at (2,0) {$\E(X,X)$};
        \node (Y) at (0,0) {$\E(X,X)$};
        \draw[double] (X) to node [sloped] {$\shortmid$} (B);
        \draw[->] (X) to node[left] {{\scriptsize $U^\D_X$}} (Y');
        \draw[->] (B) to node[right] {{\scriptsize $U^\D_{X}$}} (Z');
        \draw[double] (Y') to node [sloped] {$\shortmid$} (Z');
        \draw[double] (Y') to (Y);
        \draw[double] (Z') to (Z);
        \draw[->] (Y) to node [sloped] {$\shortmid$} node[below] {{\scriptsize $\E(id^\V_X,id^\V_X)$}} (Z);
        \node () at (1,2.25) {$U^\eD_{U^\D_X}$};
        \node () at (1,0.75) {$I^\D_X$};
      \end{tikzpicture}\\
      \hspace*{-2cm}
      \begin{tikzpicture}
        \node (A) at (0,3) {$\E(X,Y)\otimes\E(Y,Z)$};
        \node (A') at (0,1.5) {$\E(X,Y)\otimes\E(Y,Z)$};
        \node (B) at (5,3) {$\E(X,Y)\otimes\E(Y,Z)$};
        \node (B') at (5,1.5) {$\E(X,Y)\otimes\E(Y,Z)$};
        \node (Z) at (5,0) {$\E(X,Z)$};
        \node (Y) at (0,0) {$\E(X,Z)$};
        \draw[double] (A) to node [sloped] {$\shortmid$}  (B);
        \draw[->] (A') to node [sloped] {$\shortmid$} node[above] {{\scriptsize $\E(id^\V_X,id^\V_Y)\otimes\E(id^\V_Y,id^\V_Z)$}} (B');
        \draw[double] (A) to (A');
        \draw[double] (B) to (B');
        \draw[->] (A') to node[left] {{\scriptsize $\odot^\E_{X,Y,Z}$}} (Y);
        \draw[->] (B') to node[right] {{\scriptsize $\odot^\E_{X,Y,Z}$}} (Z);
        \draw[->] (Y) to node [sloped] {$\shortmid$} node[below] {{\scriptsize $\E(id^\V_X,id^\V_Z)$}} (Z);
        \node () at (2.5,2.25) {$I^\D_{X,Y}\otimes I^\D_{X,Y}$};
        \node () at (2.5,0.75) {$\odot^\D_{id^\V_X,id^\V_Y,id^\V_Z}$};
      \end{tikzpicture}\\
      \hspace*{2cm}
      \begin{tikzpicture}
        \node () at (-2,1.5) {$=$};
        \node (X) at (0,3) {$\E(X,Y)\otimes\E(Y,Z)$};
        \node (B) at (5,3) {$\E(X,Y)\otimes\E(Y,Z)$};
        \node (Z') at (5,1.5) {$\E(X,Z)$};
        \node (Y') at (0,1.5) {$\E(X,Z)$};
        \node (Z) at (5,0) {$\E(X,Z)$};
        \node (Y) at (0,0) {$\E(X,Z)$};
        \draw[double] (X) to node [sloped] {$\shortmid$}  (B);
        \draw[->] (X) to node[left] {{\scriptsize $\odot^\D_{X,Y,Z}$}} (Y');
        \draw[->] (B) to node[right] {{\scriptsize $\odot^\D_{X,Y,Z}$}} (Z');
        \draw[double] (Y') to node [sloped] {$\shortmid$} (Z');
        \draw[double] (Y') to (Y);
        \draw[double] (Z') to (Z);
        \draw[->] (Y) to node [sloped] {$\shortmid$} node[below] {{\scriptsize $\E(id^\V_X,id^\V_Z)$}} (Z);
        \node () at (2.5,2.25) {$U^\eD_{\odot^\D_{X,Y,Z}}$};
        \node () at (2.5,0.75) {$I^\D_{X,Y}$};
      \end{tikzpicture}\\
      \begin{tikzpicture}
        \node (A) at (0,3) {$\1$};
        \node (C) at (3,3) {$\1$};
        \node (X) at (0,0) {$\E(X_1,X_1)$};
        \node (Z) at (3,0) {$\E(X_3,X_3)$};
        \draw[double] (A) to node [sloped] {$\shortmid$} (C);
        \draw[->] (A) to node[left] {{\scriptsize $U^\D_{X_1}$}} (X);
        \draw[->] (C) to node[right] {{\scriptsize $U^\D_{X_3}$}} (Z);
        \draw[->] (X) to node [sloped] {$\shortmid$} node[below] {{\scriptsize $\E(f;g,f;g)$}} (Z);
        \node () at (1.5,1.5) {$U^\D_{f;g}$};
        \node () at (4,1.5) {$=$};
      \end{tikzpicture}
      \begin{tikzpicture}
        \node (A) at (0,3) {$\1$};
        \node (B) at (2,3) {$\1$};
        \node (C) at (4,3) {$\1$};
        \node (X) at (0,1.5) {$\E(X_1,X_1)$};
        \node (Y) at (2,1.5) {$\E(X_2,X_2)$};
        \node (Z) at (4,1.5) {$\E(X_3,X_3)$};
        \node (X') at (0,0) {$\E(X_1,X_1)$};
        \node (Z') at (4,0) {$\E(X_3,X_3)$};
        \draw[double] (A) to node [sloped] {$\shortmid$} (B);
        \draw[double] (B) to node [sloped] {$\shortmid$} (C);
        \draw[->] (A) to node[left] {{\scriptsize $U^\D_{X_1}$}} (X);
        \draw[->] (B) to node[right] {{\scriptsize $U^\D_{X_2}$}} (Y);
        \draw[->] (C) to node[right] {{\scriptsize $U^\D_{X_3}$}} (Z);
        \draw[->] (X) to node [sloped] {$\shortmid$} node[below] {{\scriptsize $\E(f,f)$}} (Y);
        \draw[->] (Y) to node [sloped] {$\shortmid$} node[below] {{\scriptsize $\E(g,g)$}} (Z);
        \draw[double] (X) to (X');
        \draw[double] (Z) to (Z');
        \draw[->] (X') to node [sloped] {$\shortmid$} node[below] {{\scriptsize $\E(f;g,f;g)$}} (Z');
        \node () at (1,2.25) {$U^\D_f$};
        \node () at (3,2.25) {$U^\D_g$};
        \node () at (2,.75) {$C^\D_{f,f,g,g}$};
      \end{tikzpicture}\\
      \hspace{-1cm}
      \begin{tikzpicture}
        \node (A) at (0,3) {$\E(X_1,Y_1)\otimes\E(Y_1,Z_1)$};
        \node (B) at (4,3) {$\E(X_2,Y_2)\otimes\E(Y_2,Z_2)$};
        \node (C) at (8,3) {$\E(X_3,Y_3)\otimes\E(Y_3,Z_3)$};
        \node (A') at (0,1.5) {$\E(X_1,Y_1)\otimes\E(Y_1,Z_1)$};
        \node (C') at (8,1.5) {$\E(X_3,Y_3)\otimes\E(Y_3,Z_3)$};
        \node (Y) at (0,0) {$\E(X_1,Z_1)$};
        \node (Z) at (8,0) {$\E(X_3,Z_3)$};
        \draw[->] (A) to node [sloped] {$\shortmid$} node[above] {{\scriptsize $\E(f,g)\otimes\E(g,h)$}} (B);
        \draw[->] (B) to node [sloped] {$\shortmid$} node[above] {{\scriptsize $\E(f',g')\otimes\E(g',h')$}} (C);
        \draw[double] (A') to (A);
        \draw[double] (C') to (C);
        \draw[->] (A') to node [sloped] {$\shortmid$} node[above] {{\scriptsize $\E(f;f',g;g')\otimes\E(g;g',h;h')$}} (C');
        \draw[->] (A') to node[left] {{\scriptsize $\odot^\D_{X_1,Y_1,Z_1}$}} (Y);
        \draw[->] (C') to node[right] {{\scriptsize $\odot^\D_{X_3,Y_3,Z_3}$}} (Z);
        \draw[->] (Y) to node [sloped] {$\shortmid$} node[below] {{\scriptsize $\E(f;f',h;h')$}} (Z);
        \node () at (4,2.25) {$C^\D_{f,g,f',g'}\otimes C^\D_{g,h,g',h'}$};
        \node () at (4,.75) {$\odot^\D_{f;f',g;g',h;h'}$};
      \end{tikzpicture}\\
      \begin{tikzpicture}
        \node () at (-2,1.5) {$=$};
        \node (A) at (0,3) {$\E(X_1,Y_1)\otimes\E(Y_1,Z_1)$};
        \node (B) at (4,3) {$\E(X_2,Y_2)\otimes\E(Y_2,Z_2)$};
        \node (C) at (8,3) {$\E(X_3,Y_3)\otimes\E(Y_3,Z_3')$};
        \node (X) at (0,1.5) {$\E(X_1,Z_1)$};
        \node (Y) at (4,1.5) {$\E(X_2,Z_2)$};
        \node (Z) at (8,1.5) {$\E(X_3,Z_3)$};
        \node (X') at (0,0) {$\E(X_1,Z_1)$};
        \node (Z') at (8,0) {$\E(X_3,Z_3)$};
        \draw[->] (A) to node [sloped] {$\shortmid$} node[above] {{\scriptsize $\E(f,g)\otimes\E(g,h)$}} (B);
        \draw[->] (B) to node [sloped] {$\shortmid$} node[above] {{\scriptsize $\E(f',g')\otimes\E(g',h')$}} (C);
        \draw[->] (A) to node[left] {{\scriptsize $\odot^\E_{X_1,Y_1,Z_1}$}} (X);
        \draw[->] (B) to node[left] {{\scriptsize $\odot^\E_{X_2,Y_2,Z_2}$}} (Y);
        \draw[->] (C) to node[right] {{\scriptsize $\odot^\E_{X_3,Y_3,Z_3}$}} (Z);
        \draw[->] (X) to node [sloped] {$\shortmid$} node[below] {{\scriptsize $\E(f,h)$}} (Y);
        \draw[->] (Y) to node [sloped] {$\shortmid$} node[below] {{\scriptsize $\E(f',h')$}} (Z);
        \draw[double] (X') to (X);
        \draw[double] (Z') to (Z);
        \draw[->] (X') to node [sloped] {$\shortmid$} node[below] {{\scriptsize $\E(f;f',h;h')$}} (Z');
        \node () at (2,2.25) {$\odot^\D_{f,g,h}$};
        \node () at (6,2.25) {$\odot^\D_{f',g',h'}$};
        \node () at (4,0.75) {$C^\D_{f,h,f',h'}$};
      \end{tikzpicture}
    \end{center}
  \item with a tensorial structure
    \begin{align*}
      (U^\D_{f}\otimes id^\eA_{\E(f,g)});(\odot^\E_{f,f,g}) &= id^\eA_{\E(f,g)} \\
      (id^\eA_{\E(f,g)}\otimes U^\D_{g});(\odot^\D_{f,g,g}) &= id^\eA_{\E(f,g)} \\
      ((\odot^\D_{f,g,h})\otimes id^\eA_{\E(h,k)});(\odot^\D_{f,h,k}) &= (id^\eA_{\E(f,g)}\otimes(\odot^\D_{g,h,k}) );(\odot^\D_{f,g,k}) \\
    \end{align*}
  \item and a distribution :
    \begin{align*}
      (\E(f,g)\odot^\D\E(g,h));(\E(f',g')\odot^\D\E(g',h')) &= (\E(f,g);\E(f',g'))\odot^\D(\E(g,h);\E(g',h'))
    \end{align*}
  \end{itemize}

    In addition, we call vertical category the $\eV$-enriched category $\H$ :
    \begin{itemize}
    \item which objects are those of the vertical category : $|\H|=|\V|$,
    \item which morphisms are the objects of the arrow category : 
      $$\H(X,Y) = \E(X,Y) \in |\eH|=|\eV|$$ 
    \item and which unit and composition are given by $U$ and $\odot$ on objects.
    \end{itemize}
\end{definition}

Similarly to the slice over $\1$ that give the underlying non-enriched category, the Grothendieck construction on a vertically enriched double category gives the underlying non-enriched double category :

\begin{theorem}
  Let $\eD =(\eA,\eH,\eV,U^\eD,\odot^\eD)$ be a monoidal double category and $\D=(\E,\H,V,U^\D,\odot^\D)$ a $\eD$-enriched double category.\\
  The underlying non-enriched double category $\int\D$ is given by :
  \begin{itemize}
  \item the vertical category do not change $V:=\V$,
  \item the arrow category is the Grothendieck construction $\A=\int\E$, spelled out, it is :
    $$\left|\int\E\right| := \left\{(X,Y,u) \middle| X,Y\in|\V|, u\in\eV(\1,\E(X,Y))\right\},$$
    $$\left(\int\E\right)((X,Y,u),(X',Y',v)) := \left\{(f,g,\phi) \middle| \begin{matrix}f\in \V(X,X'),\ g\in\V(Y,Y'),\\ \phi\in\eA(U^\eD_\1,\E(f,g)),\\ S(\phi)=u,\ T(\phi)=v\end{matrix}\right\},$$
  \item the source and target are given by the Grothendieck projection :
    $$S(X,Y,u):=X,\quad S(f,g,\phi):=f,\quad T(X,Y,u):=Y,\quad S(f,g,\phi):=g,$$
  \item $U : \V \rightarrow \int\E$ is given by 
    $$U_X := (X,X,U_X), \quad U_f:=(f,f,U^\D_f)$$
  \item $(\_\odot\_) : T\vee S \rightarrow \int\E$ is given by :
    $$ (X,Y,u)\odot(Y,Z,v) := (X,Z,(\rho;u\otimes v;\odot^\D)) $$
    $$ (f,g,\phi)\odot(g,h,\psi) := (f,h,(U^\eD_\rho;(\phi\otimes\psi);\odot_{f,g,h})) $$
  \end{itemize}
\end{theorem}

Remarks : 
\begin{itemize}
\item For $(\Set,\Set)$-enriched double categories, the Grothendieck construction is known to give only framed categories, but this is not true for any enrichment, in particular, we will see that the Grothendieck construction over $(\Span,\Set)$-enriched double categories is an equivalence of category.
\item The transversal double category cannot be easily defined as we did not define enrichments where vertical category are enriched.
\item The diagonal and globular categories can be defined as (2-)categories by looking at the underlying non-enriched double category. But enriched versions are unknown.
\end{itemize}

\section{Examples}

\begin{lemma}
  Thin double categories are exactly horizontally $(\Rel,\Set)$-enriched double category.
\end{lemma}
\begin{proof}
  By definitions, a $\Rel$-enriched double category $\D$ is given by 
  \begin{itemize}
  \item a vertical category $\V$
  \item a lax functor :
    $$ \E : \V\times\V \rightarrow \Rel $$
    whose lax-ness spells out as :
    $$ id^\Rel_{\E(X,Y)} \subseteq \E(id^\V_X,id^\V_Y) $$
    $$ \E(f,g);\E(f';g') \subseteq \E(f;f',g;g') $$
  \item with
    $$ U^\D_X \in \E(X,X)$$
    $$ (\_\odot\_) : \E(X,Y)\times \E(Y,Z) \rightarrow \E(X,Z) $$
    natural in $X$, $Y$ and $Z$ in the sens that for any $f,g,h$ :
    $$ (U^\D_X,U^\D_{X'}) \in \E(f,f) $$
    $$ \bigl(\ (u,u')\in D(f,g) \quad\text{and}\quad (v,v')\in D(g,h)\ \bigr)\quad \Rightarrow\quad (u\odot v,u'\odot v')\in D(f,h) $$
  \end{itemize}
  Let $\D=(\E,\V,\odot,U)$ such a category. The following is a thin double category :
  \begin{itemize}
  \item $|\A| := \biguplus_{X,Y\in\|V|}\E(X,Y)$
  \item $\A(u,v) := \{(f,g) \mid (u,v)\in \E(f,g)\}$
  \item $id_u^{\A} := (id_X,id_Y) \in \A(u,u)$ since $(u,u)\in id^{\Rel}_{\E(X,Y)}\subseteq \E(id_X,id_X)$
  \item $(f,g);(f',g') := (f;f',g;g')$
  \item $S(X,Y,u)=X$, $T(X,Y,u)=Y$, $S(f,g)=f$, $T(f,g)=g$
  \item $U_X \in \E(X,X)\subseteq |\A|$,  $U_f:=\{(f,f)\}\in\A(U_X,U_Y)$
  \item $f \odot g \in\E(X,Z)\subseteq |\A|$, $\phi\odot\psi := (S(\phi),T(\psi))$
  \end{itemize}
  Given a thin double category $\D$, the following is a $\Rel$-enriched double category :
  \begin{itemize}
  \item $\E(X,Y) := \{u \mid S(u)=X, T(u) = Y\}$,
  \item $\E(f,g) := \{(u,v) \mid \exists \phi\in \A(u,v)\mid S(\phi)=f, T(\phi)=g\}$
  \item $(u,u)\in \E(id_X,id_Y)$ since $id_u^\A\in\A(u,u)$ and $S,T$ are functors,
  \item if $(u,v)\in\E(f,g)$ and $(v,w)\in\E(f',g')$, then there are $\phi\in\A(u,v)$, $\psi\in\A(v,w)$ with appropriate source and targets, and thus  $(u,w)\in\E(f;f',g;g')$ since $\phi;\psi \in\A(u,w)$ with appropriate source and target (by functoriality of $S,T$)
  \end{itemize}
  Conversely, given a $\Rel$-enriched double category, the Grothendieck construction give a double category such that for any given $u,v,f,g$ : 
  \begin{align*}
    &\left\{\chi\in\left(\int\E\right)((X,Y,u),(X',Y',v)) \middle| S(\chi)=f, T(\chi)=g\right\}\\
    &= \left\{(f,g,\phi)\in\left(\int\E\right)((X,Y,u),(X',Y',v)) \right\}\\
    &= \left\{(f,g,\phi)\middle| \phi\in\mathcal\eA(U^\eD_\1,\E(f,g)), S^\eD(\phi)=u,S^\eD(\phi)=v\right\}\\
    &= \left\{(f,g,*)\middle| v_*\ge u_*;\E(f,g)\right\}
    \end{align*}
    which is either empty or a singleton.
    \todo[inline]{show that this is an equivalence of category}
\end{proof}

\begin{lemma}
  Double categories are exactly internally $(\Span,\Set)$-enriched double categories, where $\Span$ is the category of sets and multirelations (i.e. maps $X\times Y\rightarrow \Set$).
\end{lemma}
\begin{proof}
\todo[inline]{Todo...}
\end{proof}

\begin{lemma}
  Framed categories are exactly internally $(\Set,\Set)$-enriched double categories.
\end{lemma}
\begin{proof}
\todo[inline]{Todo...}
\end{proof}

\begin{lemma}
  Externally $\C$-enriched categories are exactly internally $(\C,\Span(\C))$-enriched category (up to a transposition).
\end{lemma}
\begin{proof}
\todo[inline]{Todo...}
\end{proof}

\section{Enriched double functor}

When considering enriched double functors, there seems to be a choice to make : should we associate the fibers of the functor to horizontal morphisms or to a vertical morphisms of $\eD$ ? Since we do not make such an arbitrary choice, we will take a third option : considering that $\eD$ is a triple category and use ``depth'' morphisms. This strictly include both of the other choices by having the ``depth''-category equating the vertical or horizontal categories.

In fact we claim that such an approach would make sens for any lax (normal) functor : it is always free to consider that the enrichment of the functor and those of the categories can follow different dimensions of a double category (or any tuple category if one want to work with different kinds of functors).

\begin{definition}
  Let $\eD =(\eA,\eV,\eH,\mathcal P)$ be a monoidal triple category, $\eD^{v,h}:=(\eA^{v,h},\eV,\eH)$ its (x,y)-plan.

  Let $\D_1=(\E_1,\V_1,\H_1,U^{\D_1},\odot^{\D_1})$ and $\D_2=(\E_2,\V_2,\H_2,U^{\D_2},\odot^{\D_2})$ two $\eD^{v,h}$-enriched double categories.

  A $\eD$-enriched functor $F : \D_1\rightarrow \D_2$ is given by :
  \begin{itemize}
  \item a functor on vertical categories :
    $$ F_v : \V_1 \rightarrow \V_2 $$
  \item A functor for the arrow category :
    $$ F_a : \V_1\times\V_1 \rightarrow \eA^{\mathcal P,\eH} $$
    agreeing with each other :
    $$ F_a;S^{\eD^{\mathcal P,\eH}}=\E_1 \quad\quad F_a;T^{\eD^{\mathcal P,\eH}} =(F_v\times F_v);\E_2 $$
    i.e. :\vspace{-2em}
    \begin{center}
      \begin{tikzpicture}
        \node (X) at (0,2) {$\E_1(X,Y)$};
        \node (Y) at (4,2) {$\E_2(FX,FY)$};
        \node (U) at (0,0) {$\E_1(X',Y')$};
        \node (V) at (4,0) {$\E_2(FX',FY')$};
        \draw[->] (X) to node [sloped] {$\circ$} node[auto] {{\scriptsize $F(X,X)$}}  (Y);
        \draw[->] (X) to node [sloped] {$\shortmid$} node[left] {{\scriptsize $\E_1(f,g)$}}   (U);
        \draw[->] (Y) to node [sloped] {$\shortmid$} node[auto] {{\scriptsize $\E_2(Ff,Fg)$}}  (V);
        \draw[->] (U) to node [sloped] {$\circ$} node[below] {{\scriptsize $F(X',Y')$}} (V);
        \node () at (2,1) {$F(f,g)$};
      \end{tikzpicture}
    \end{center}
    and with the laxness of $\E_1$ and $\E_2$ :
    
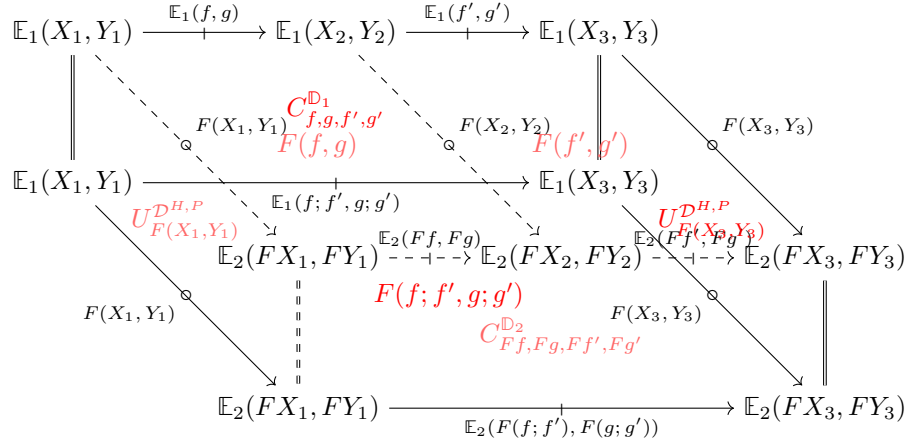
\begin{center}
      \begin{tikzpicture}
        \node (X) at (0,5) {$\E_1(X_1,Y_1)$};
        \node (X2) at (3.5,5) {$\E_1(X_2,Y_2)$};
        \node (Y) at (7,5) {$\E_1(X_3,Y_3)$};
        \node (U) at (0,3) {$\E_1(X_1,Y_1)$};
        \node (V) at (7,3) {$\E_1(X_3,Y_3)$};
        \node (X') at (3,2) {$\E_2(FX_1,FY_1)$};
        \node (X2') at (6.5,2) {$\E_2(FX_2,FY_2)$};
        \node (Y') at (10,2) {$\E_2(FX_3,FY_3)$};
        \node (U') at (3,0) {$\E_2(FX_1,FY_1)$};
        \node (V') at (10,0) {$\E_2(FX_3,FY_3)$};
        \draw[->] (X) to node (i1) [sloped] {$\shortmid$} node[auto] {{\scriptsize $\E_1(f,g)$}} (X2);
        \draw[->] (X2) to node (i5) [sloped] {$\shortmid$} node[auto] {{\scriptsize $\E_1(f',g')$}} (Y);
        \draw[double] (X) to  (U);
        \draw[double] (Y) to  (V);
        \draw[->] (U) to node (j1) [sloped] {$\shortmid$} node[below] {{\scriptsize $\E_1(f;f',g;g')$}} (V);
        \draw[->,dashed] (X') to node (i2) [sloped] {$\shortmid$} node[auto] {{\scriptsize $\E_2(Ff,Fg)$}} (X2');
        \draw[->,dashed] (X2') to node (j5) [sloped] {$\shortmid$} node[auto] {{\scriptsize $\E_2(Ff',Fg')$}} (Y');
        \draw[double,dashed] (X') to  (U');
        \draw[double] (Y') to  (V');
        \draw[->] (U') to node (j2) [sloped] {$\shortmid$} node[below] {{\scriptsize $\E_2(F(f;f'),F(g;g'))$}} (V');
        \draw[->,dashed] (X) to node [sloped] (i3) {$\circ$} node[auto] {{\scriptsize $F(X_1,Y_1)$}} (X');
        \draw[->,dashed] (X2) to node [sloped] {$\circ$} node[auto] {{\scriptsize $F(X_2,Y_2)$}} (X2');
        \draw[->] (Y) to node [sloped] (i4) {$\circ$}  node[auto] {{\scriptsize $F(X_3,Y_3)$}} (Y');
        \draw[->] (U) to node [sloped] (j3) {$\circ$} node[below left] {{\scriptsize $F(X_1,Y_1)$}}  (U');
        \draw[->] (V) to node [sloped] (j4) {$\circ$} node[below left] {{\scriptsize $F(X_3,Y_3)$}}  (V');
        \draw[draw=none,red] (X2) to node {$C^{\D_1}_{f,g,f',g'}$}  (j1);
        \draw[draw=none,red!60] (X2') to node {$C^{\D_2}_{Ff,Fg,Ff',Fg'}$}  (j2);
        \draw[draw=none,red] (j1) to node {$F(f;f',g;g')$} (j2);
        \draw[draw=none,red!60] (i3) to node {$U^{\eD^{H,P}}_{F(X_1,Y_1)}$} (j3);
        \draw[draw=none,red] (i4) to node {$U^{\eD^{H,P}}_{F(X_3,Y_3)}$} (j4);
        \draw[draw=none,red!60] (i1) to node {$F(f,g)$} (i2);
        \draw[draw=none,red!60] (i5) to node {$F(f',g')$} (j5);
      \end{tikzpicture}
    \end{center}
  \item A unit functor :
    $$ U^F : \V_1 \rightarrow \mathcal D^{\eH,\mathcal P,\eV} $$
    agreeing with the two others :
    $$ U^F:S^P = U^{\D_1}\quad U^F;T = F_v;U^{E_2} \quad U^F;S^V = id_{id_\1} \quad U^F;T^V=F_a$$
    i.e., $U^F_f$ is the cube :
    
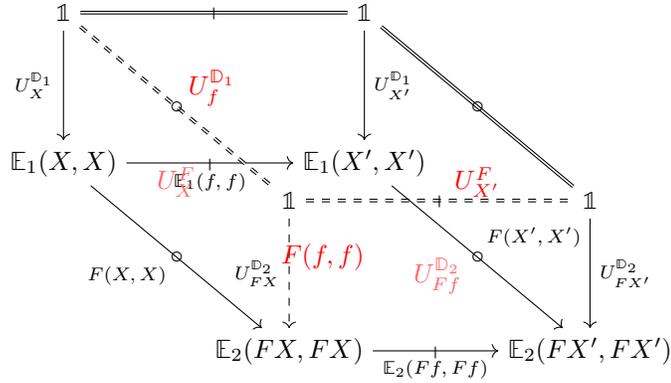
\begin{center}
      \begin{tikzpicture}
        \node (X) at (0,4.5) {$\1$};
        \node (Y) at (4,4.5) {$\1$};
        \node (U) at (0,2.5) {$\E_1(X,X)$};
        \node (V) at (4,2.5) {$\E_1(X',X')$};
        \node (X') at (3,2) {$\1$};
        \node (Y') at (7,2) {$\1$};
        \node (U') at (3,0) {$\E_2(FX,FX)$};
        \node (V') at (7,0) {$\E_2(FX',FX')$};
        \draw[double] (X) to node (i1) [sloped] {$\shortmid$}  (Y);
        \draw[->] (X) to node[left] {{\scriptsize $U_X^{\D_1}$}}   (U);
        \draw[->] (Y) to node[auto] {{\scriptsize $U_{X'}^{\D_1}$}}  (V);
        \draw[->] (U) to node (j1) [sloped] {$\shortmid$} node[below] {{\scriptsize $\E_1(f,f)$}} (V);
        \draw[double,dashed] (X') to node (i2) [sloped] {$\shortmid$}  (Y');
        \draw[->,dashed] (X') to node[left] {{\scriptsize $U_{FX}^{\D_2}$}}   (U');
        \draw[->] (Y') to node[auto] {{\scriptsize $U_{FX'}^{\D_2}$}}  (V');
        \draw[->] (U') to node [sloped] (j2) {$\shortmid$} node[below] {{\scriptsize $\E_2(Ff,Ff)$}} (V');
        \draw[double,dashed] (X) to node [sloped] (i3) {$\circ$}  (X');
        \draw[double] (Y) to node [sloped] (i4) {$\circ$}  (Y');
        \draw[->] (U) to node [sloped] (j3) {$\circ$} node[below left] {{\scriptsize $F(X,X)$}}  (U');
        \draw[->] (V) to node [sloped] (j4) {$\circ$} node[auto] {{\scriptsize $F(X',X')$}}  (V');
        \draw[draw=none,red] (i1) to node {$U^{\D_1}_{f}$}  (j1);
        \draw[draw=none,red!60] (i2) to node {$U^{\D_2}_{Ff}$} (j2);
        \draw[draw=none,red] (j1) to node {$F(f,f)$} (j2);
        \draw[draw=none,red!60] (i3) to node {$U^F_X$} (j3);
        \draw[draw=none,red] (i4) to node {$U^F_{X'}$} (j4);
      \end{tikzpicture}
    \end{center}
  \item A composition functor :
    $$ \odot^F : \V_1\times\V_1\times\V_1 \rightarrow \mathcal D^{\eH,\mathcal P,\eV} $$
    agreeing with the two others :
    $$ \odot^F;S^P = \odot^{\D_1}\quad \odot^F;T = F_v;\odot^{E_2} \quad U^F;S^V =c_2;(F_a\times F_a) \quad U^F;T^V=p_{1,3};F_a$$
    where $c_2(X,Y,Z)=((X,Y),(Y,Z))$ and $p_2(X,Y,Z)=(X,Z)$ are functors copying and erasing the second argument.\\
    In diagram, this give the cubes $\odot^F_{f,g,h}$  :
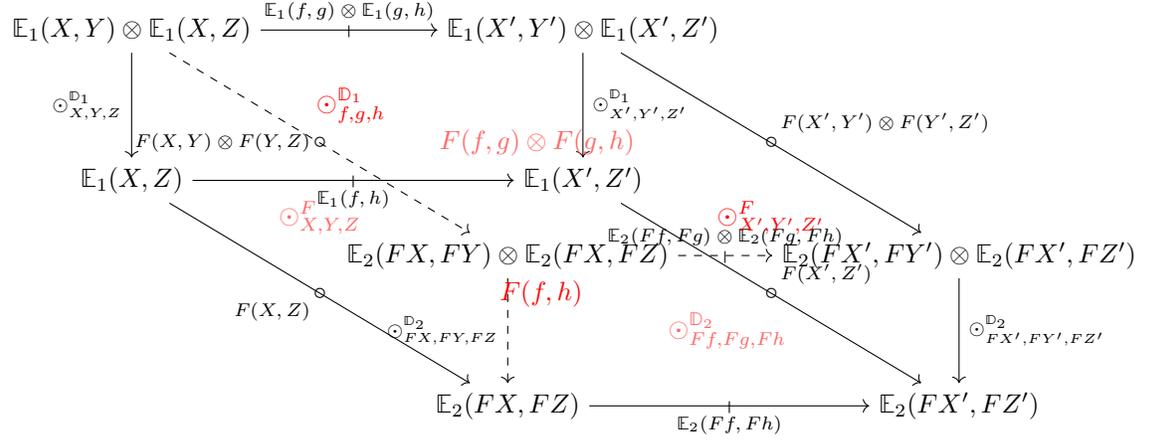
    \begin{center}
      \begin{tikzpicture}
        \node (X) at (0,5) {$\E_1(X,Y)\otimes\E_1(X,Z)$};
        \node (Y) at (6,5) {$\E_1(X',Y')\otimes\E_1(X',Z')$};
        \node (U) at (0,3) {$\E_1(X,Z)$};
        \node (V) at (6,3) {$\E_1(X',Z')$};
        \node (X') at (5,2) {$\E_2(FX,FY)\otimes\E_2(FX,FZ)$};
        \node (Y') at (11,2) {$\E_2(FX',FY')\otimes\E_2(FX',FZ')$};
        \node (U') at (5,0) {$\E_2(FX,FZ)$};
        \node (V') at (11,0) {$\E_2(FX',FZ')$};
        \draw[->] (X) to node (i1) [sloped] {$\shortmid$} node[above] {{\scriptsize $\E_1(f,g)\otimes\E_1(g,h)$}} (Y);
        \draw[->] (X) to node[left] {{\scriptsize $\odot_{X,Y,Z}^{\D_1}$}}   (U);
        \draw[->] (Y) to node[auto] {{\scriptsize $\odot_{X',Y',Z'}^{\D_1}$}}  (V);
        \draw[->] (U) to node (j1) [sloped] {$\shortmid$} node[below] {{\scriptsize $\E_1(f,h)$}} (V);
        \draw[->,dashed] (X') to node (i2) [sloped] {$\shortmid$} node[above] {{\scriptsize $\E_2(Ff,Fg)\otimes\E_2(Fg,Fh)$}} (Y');
        \draw[->,dashed] (X') to node[left] {{\scriptsize $\odot_{FX,FY,FZ}^{\D_2}$}}   (U');
        \draw[->] (Y') to node[auto] {{\scriptsize $\odot_{FX',FY',FZ'}^{\D_2}$}}  (V');
        \draw[->] (U') to node [sloped] (j2) {$\shortmid$} node[below] {{\scriptsize $\E_2(Ff,Fh)$}} (V');
        \draw[->,dashed] (X) to node [sloped] (i3) {$\circ$}  node[left] {{\scriptsize $F(X,Y)\otimes F(Y,Z)$}} (X');
        \draw[->] (Y) to node [sloped] (i4) {$\circ$} node[auto] {{\scriptsize $F(X',Y')\otimes F(Y',Z')$}} (Y');
        \draw[->] (U) to node [sloped] (j3) {$\circ$} node[below left] {{\scriptsize $F(X,Z)$}}  (U');
        \draw[->] (V) to node [sloped] (j4) {$\circ$} node[auto] {{\scriptsize $F(X',Z')$}}  (V');
        \draw[draw=none,red] (i1) to node {$\odot^{\D_1}_{f,g,h}$}  (j1);
        \draw[draw=none,red!60] (i2) to node {$\odot^{\D_2}_{Ff,Fg,Fh}$} (j2);
        \draw[draw=none,red] (j1) to node {$F(f,h)$} (j2);
        \draw[draw=none,red!60] (i1) to node {$F(f,g)\otimes F(g,h)$} (i2);
        \draw[draw=none,red!60] (i3) to node {$\odot^F_{X,Y,Z}$} (j3);
        \draw[draw=none,red] (i4) to node {$\odot^F_{X',Y',Z'}$} (j4);
      \end{tikzpicture}
    \end{center}
  \item with expected coherence axioms obtained by adding a dimensions to the lax functorial equalities of $U$ and $\odot$.
  \end{itemize}
\end{definition}

\section{Externally enriching internally enriched double categories}

\begin{definition}\ \\
  Let $\mathfrak D=(\mathfrak A,\mathfrak H,\mathfrak V,U^{\mathfrak D},\odot^{\mathfrak D})$ a monoidal double category and $\eD =(\eE,\eH,\eV,U^\eE,\odot^\eE)$ be a monoidal $\mathfrak D$-enriched double category.

  A  $(\mathfrak D,\eD)$-enriched double category $\D$ is given by :
  \begin{itemize}
  \item a vertical $\mathfrak V$-enriched category $\V$
  \item a lax $\mathfrak V$-enriched functor :
    $$ \E : \V\times\V \rightarrow \eH $$
  \end{itemize}
  \todo[inline]{todo}
\end{definition}

\end{document}